\documentclass{amsart}
\usepackage{a4wide,amssymb,color}

\newtheorem{theorem}{Theorem}
\newtheorem{proposition}[theorem]{Proposition}
\newtheorem{corollary}[theorem]{Corollary}

\newtheorem{lemma}[theorem]{Lemma}
\theoremstyle{definition}
\newtheorem{definition}[theorem]{Definition}

\newcommand{\defeq}{:=}
\newcommand{\IR}{\mathbb R}
\newcommand{\IZ}{\mathbb Z}
\newcommand{\IN}{\mathbb N}

\title{Locally compact strictly convex metric groups are abelian}
\author{Taras Banakh, Oles Mazurenko}
\email{t.o.banakh@gmail.com, oles.mazurenko@lnu.edu.ua}
\address{Ivan Franko National University of Lviv}

\begin{document}
\begin{abstract}
We show that every locally compact strictly convex metric group is abelian, thus answering one problem posed by the authors in their earlir paper. To prove this theorem we first construct the isomorphic embeddings of the real line into the strictly convex metric group using its geodesic properties and charaterization of the real line as a unique not monothetic one-parametric metrizable topological group. We proceed to show that all compact subgroups in a strictly convex metric group are trivial, which combined with the classical result of Iwasawa completes the proof of the main result. 
\end{abstract}
\maketitle

	\section{Introduction} 
	
The strict convexity is a fundamental geometric property of Banach spaces, ensuring the uniqueness of best approximations, playing a central role in duality theory, and underpinning applications in optimization, approximation theory, and fixed point theory (cf. \cite{Barbu}, \cite{BorLewis}, \cite{GK}, \cite{Ist}).
In spite of the fact that strict convexity is usually defined for normed or Banach spaces, it is a purely metric property and can be defined without involving the linear or convex structure. 

\begin{definition} A metric space $(X,d)$ is defined to be {\em strictly convex} if for any points $x,y\in X$ and any positive real numbers $a,b$ with $a+b=d(x,y)$, the intersection $B[x,a]\cap B[y,b]$ is a singleton.
\end{definition}

Here we denote by $B[x,a]\defeq\{z\in X:d(x,z)\le a\}$ the closed ball of radius $a$ centered at a point $x$ of the metric space $(X,d)$.

In the paper \cite{BM1} we proved that every strictly convex metric abelian group $G$ admits a unique multiplication $\cdot:\IR\times G\to G$ turning $G$ into a normed space over the field of real numbers, and asked whether every strictly convex metric group is abelian. In this paper we give a partial answer to this problem proving that a strictly convex metric group is abelian whenever it is locally compact or finite-dimensional.

\section{The main result}

Our study mainly considers the strict convexity in metric groups. Let us recall the definition of this mathematical structure.
	\begin{definition}
		A \emph{group} is an algebraic structure $(G, +, 0)$, consisting of a set $G$, a binary operation $+: G \times G \to G$  and an identity element $0$, satisfying the following axioms:
		\begin{enumerate}
			\item $\forall x,y,z\in G\;\;(x + y)+z = x+(y+z)$, \hfill (associativity)
			\item $\forall x\in G\;\;x+0 = x=0+x$, \hfill (identity)
			\item $\forall x\in G\;\exists y\in G\;\;x+y=0=y+x$. \hfill (inverse)
		\end{enumerate} 
		If, in addition, $x+y = y+x$ for all $x,y \in G$, then $(G,+, 0)$ is called an \emph{abelian group}.
	\end{definition}
	\begin{definition}
		A \emph{metric group} is a group $(G, +, 0)$ equipped with a metric $d: G \times G \to \mathbb{R}$, which is translation invariant in the sense that $d(x + c, y + c) = d(x,y)=d(c+x,c+y)$ for all points $x, y, c \in G$. The metric $d$ can be uniquely recovered from the {\em norm} $$\|\cdot\|:G\to\IR,\quad \|\cdot\|:x\mapsto \|x\|\defeq d(0,x),$$
		generated by this metric.
	\end{definition}
	\begin{definition}
		A topological space $X$ is called {\em locally compact} if for every $x \in X$ there exists an open neighborhood $U$ of $x$ such that the closure $\overline{U}$ is compact.
	\end{definition}

	The main result of this paper is the following theorem.
	\begin{theorem}\label{BMT:1}
		Every locally compact strictly convex metric group is abelian.
	\end{theorem}
	Combined with the main result in \cite{BM1}, this implies
	\begin{corollary}
		Every locally compact strictly convex metric group is a finite-dimensional normed space over the field of real numbers.
	\end{corollary}
	Theorem~\ref{BMT:1} will be proved in Section~\ref{s:proof-main} after some preliminary work made in Sections~\ref{s3}--\ref{s6}.	
	
	

\section{Multiplication in strictly convex metric groups}\label{s3}
In this section, we shall introduce the {\em metric multiplication} and the {\em algebraic multiplication} in a  strictly convex metric group $(G, +, 0, d)$, investigate their properties and interplay.
 
\begin{definition}
	Let $(X,d), (Y, \rho)$ be metric spaces. A map $f: X \to Y$ is an \emph{isometry} if \linebreak $\rho(f(x), f(y)) = d(x,y)$ for all points $x, y \in X$.
\end{definition}
\begin{definition}
	A metric space $(X,d)$ is \emph{geodesic} if for all $x,y \in X$ there exists a unique isometry $\gamma: [0, d(x,y)] \subseteq \mathbb{R} \to X$ such that $\gamma(0) = x$ and $\gamma(d(x,y)) = y$.
\end{definition}
It is known \cite{BM1} that strictly convex metric spaces are geodesic. For every $x \in G$ let \linebreak $\gamma: [0, \|x\|] \to G$ be the unique isometry with $\gamma(0) = 0$ and $\gamma(\|x\|) = x.$ For every $t \in [0,1] \subseteq \mathbb{R}$ put $t*x := \gamma(t\|x\|).$ The defined binary operation $*:[0,1]\times G\to G$, $*:(t,x)\mapsto t*x$ will be called the {\em metric multiplication}. Witness some properties of this operation introduced in the following Lemmas.
\begin{lemma}\label{BML:1}
	For all $t,v\in [0,1]$ and $x\in G$ we have $\|t*x-v*x\|=d(t*x,v*x)=|t-v|\cdot\|x\|$.
\end{lemma}
\begin{proof}
	Since $\gamma: [0,\|x\|] \to G$ is an isometry, we obtain 
	$$\|t*x - v*x\| = d(t*x, v*x) = d(\gamma(t\|x\|), \gamma(v\|x\|)) = |t\|x\|-v\|x\|| = |t-v| \cdot \|x\|.$$
\end{proof}
\begin{corollary}\label{BMC:1}
	For all $t \in [0,1]$ and $x \in G$ we have $\|t *x\| = t\|x\|.$
\end{corollary}

\begin{lemma}\label{BML:2}
	For all $t,v \in [0,1]$ and $x \in G$ we have $t * (v * x) = (tv)*x$.
\end{lemma}
\begin{proof}
	For the element $t*(v*x)$ we have $\|t*(v*x)\| = t\|v*x\| = tv\|x\|$ and 
	$$\|t*(v*x) -  v*x\| = (1-t) \|v*x\| = (1-t)v \|x\|,$$
	by Corollary \ref{BMC:1} and Lemma \ref{BML:1}.
	Similarly, for the element $(tv)*x$ we have $\|(tv)*x\| = tv\|x\|$ and 
	$$\|(tv)*x - v*x\| = (v-tv) \|x\| = (1-t)v \|x\|.$$
	Since $tv \|x\| + (1-t)v \|x\| = v \|x\| = \|vx\| = d(0, vx)$, the element $t * (v * x)$ coincides with $(tv)*x$ by the strict convexity of the metric space $(G,d)$.
\end{proof}
\begin{lemma}\label{BML:3}
	For all $x \in G$ we have $\frac{1}{2} * x + \frac{1}{2} * x = x.$
\end{lemma}
\begin{proof}
	Let $c := ||\frac{1}{2} * x\| = \frac{1}{2}\|x\| = \|x - \frac{1}{2} * x \| $ by Lemma \ref{BML:1} and Corollary \ref{BMC:1}.
	Define $z := x - \frac{1}{2} * x \in G$. Then $\|z\| = c$
	and 
	$$\|x - z\| = d(x, z)= d(x, x - \tfrac{1}{2} * x) = d(0, - \tfrac{1}{2} * x) = \|\tfrac{1}{2} * x\| = c,$$
	by the translation invariance of the metric. Hence, by the strict convexity of the metric space $(G, d)$, it follows that $\frac{1}{2} * x = z = x - \frac{1}{2} * x$, which implies $\frac{1}{2} * x + \frac{1}{2} * x = x.$
\end{proof}

Let us define the operation $\cdot: \IZ \times G \to G$, $\cdot: (n, x) \mapsto n \cdot x$ by the recursive formulas: $0\cdot x=0$, $(n+1)\cdot x=n\cdot x+x$, and $-(n+1)\cdot x = -n \cdot x - x$ for all $n\in\IN\cup\{0\}$. We will call this operation the \emph{algebraic multiplication}.

Let $H = \{\frac{1}{2^n}: n \in \mathbb{N} \cup \{0\}\}$. We shall prove that the set $H * x = \{\frac{1}{2^n} * x: n \in \IN \cup \{0\}\}$ is commutative for all $x \in G$. This will be done using the following fact for the defined multiplication operations.
\begin{lemma}\label{BML:4}
	For all $n,m \in \IN, n \leq m$ and $x \in G$ we have $2^n \cdot (\frac{1}{2^m} * x) = \frac{1}{2^{m-n}} * x.$
\end{lemma}
\begin{proof}
	Fix $m \in \mathbb{N}$.We prove the claim by induction on $n \leq m$. If $n = 1$, then by Lemma \ref{BML:2} we have $\frac{1}{2} * (\frac{1}{2^{m-1}} * x) = \frac{1}{2^m} * x$. This implies $2 \cdot (\frac{1}{2^m} * x) = \frac{1}{2^{m-1}} * x$ by Lemma \ref{BML:3}. Assume the statement holds for some $k \in \mathbb{N}$ with $k < m$, that is $2^k \cdot (\frac{1}{2^m} * x) = \frac{1}{2^{m-k}} * x$. Then, using the base case and the induction hypothesis, we obtain $2^{k+1} \cdot (\frac{1}{2^m} * x) = 2 \cdot (\frac{1}{2^{m-k}} * x) = \frac{1}{2^{m-k-1}} * x$, which completes the induction.
\end{proof}

\begin{proposition}
	For all $x \in G$ the subset $H * x = \{\frac{1}{2^n} * x: n \in \mathbb{N} \cup \{0\}\}$ is commutative.
\end{proposition}
\begin{proof}
	Fix $x \in G$ and take $n,m \in \mathbb{N} \cup \{0\}$. We lose no generality by assuming that $n \leq m$. Then by Lemma $\ref{BML:4}$, definiton of the algebraic multiplication and the associativity of $+$ on $G$, we obtain
	$$\frac{1}{2^n} * x + \frac{1}{2^m} * x = 2^{m-n} \cdot \Big(\frac{1}{2^m} * x\Big) + \frac{1}{2^m} * x = \frac{1}{2^m} * x + 2^{m-n} \cdot \Big(\frac{1}{2^m} * x\Big) = \frac{1}{2^m} * x + \frac{1}{2^n} * x.$$
	
\end{proof}

\begin{corollary}
	For all $x \in G$ the subgroup $H_x = \langle H * x \rangle$ is abelian.
\end{corollary}

\section{Every subgroup $H_x$ is a $\IZ[\frac{1}{2}]$-module}\label{s4}
It is well-known that for all $x \in G$ the abelian group $H_x$ is a $\IZ$-module with the algebraic multiplication. In this section, we shall endow the group $H_x$ with a structure of a module over the ring $\IZ[\frac12]=\{\frac m{2^n}:m\in\IZ,\;n\in\IN\cup\{0\}\}$ of dyadic fractions.

\begin{definition}
	An additive group $(G, +, 0)$ is ({\em uniquely}) \emph{$2$-divisible} if for every $x \in G$ there exists a (unique) element $y \in G$ such that $y+y = x$. 
\end{definition}

\begin{definition} An element $x$ of a group $(G,+,0)$ is defined to have {\em order} $2$ if $x+x=0\ne x$.
\end{definition}

\begin{proposition}\label{BMP:1}
	A strictly convex metric group $(G, +, 0, d)$ has no elements of order $2$.
\end{proposition}
\begin{proof}
	See Lemma 2 in \cite{BM1}.
\end{proof}
\begin{lemma}\label{BML:5}
	For all $x \in G$ the subgroup $H_x$ is $2$-divisible.
\end{lemma}
\begin{proof}
	Take arbitrary $x \in G$ and $y \in H_x$. Since $H_x = \langle H * x \rangle$, the element $y$ can be represented as $y = \sum_{k=0}^m a_k \cdot (\frac{1}{2^{n_k}} * x)$, where $a_k \in \mathbb{Z}, n_k \in \mathbb{N} \cup \{0\}$ for all $k \in \{0, \dots m\}, m \in \mathbb{N}.$ Consider $z := \sum_{k=0}^m a_k \cdot (\frac{1}{2^{n_k + 1}} * x)$ and let us show that it satisfies the required equality using a $\IZ$-module properties of $H_x$ and Lemma \ref{BML:4}.
	$$z + z =  \sum_{k=0}^m a_k \cdot (\frac{1}{2^{n_k + 1}} * x) +  \sum_{k=0}^m a_k \cdot (\frac{1}{2^{n_k + 1}} * x) =  \sum_{k=0}^m a_k \cdot (2 \cdot \frac{1}{2^{n_k + 1}} * x) =  \sum_{k=0}^m a_k \cdot (\frac{1}{2^{n_k}} * x) = y.$$
\end{proof}

\begin{proposition}\label{BMP:2}
	For all $x \in G$ the subgroup $H_x$ is uniquely $2$-divisible.
\end{proposition}
\begin{proof}
	Fix $x \in G$. The subgroup $H_x$ is $2$-divisible by Lemma \ref{BML:5}. Since $G$ has no elements of order $2$ by Proposition \ref{BMP:1}, neither does $H_x$. If $z, z' \in H_x$ satisfy $2z = 2z'$, then $2(z-z') = 0$ by commutativity. The absence of elements of order $2$ forces $z = z'$. Hence, $H_x$ is uniquely $2$-divisible.
\end{proof}

\begin{corollary}
	For all $x \in G$ the subgroup $H_x$ is a $\IZ[\frac{1}{2}]$-module.
\end{corollary}
\begin{proof}
	It follows from Proposition \ref{BMP:2} using Proposition 2 in \cite{BM1}.
\end{proof}

As a result, for all $x \in G$ we obtained the operation $\circ_x: \IZ[\frac{1}{2}]  \times H_x \to H_x$, $\circ_x: (t,x) \mapsto t \circ_x x$ such that $(H_x, +, 0, \circ_x)$ is a $ \IZ[\frac{1}{2}]$-module. Since $H_x$ is a $\IZ$-module with the algebraic multiplication, we already know that $n \cdot y = n \circ_x y$ for all $y \in H_x$ and $n \in \IZ$. We shall now introduce the connection between $\circ_x$ and $*$. This will be done using the following Lemma. 
\begin{lemma}\label{BML:6}
	For all $n,m \in \IN, n \leq 2^m$ and $x \in G$ we have $n \cdot (\frac{1}{2^m} * x) = \frac{n}{2^{m}} * x.$
\end{lemma}
\begin{proof}
	For element $\frac{n}{2^m}*x$ we have $\|\frac{n}{2^m}*x\| = \frac{n}{2^m}\|x\|$ and $\|\frac{n}{2^m}*x - x\| = (1 - \frac{n}{2^m})\|x\|$, by Lemma \ref{BML:1} and Corollary \ref{BMC:1}. 
	For element $n \cdot (\frac{1}{2^m}*x)$ we obtain inequality
	$$\|n \cdot (\tfrac{1}{2^m}*x)\| \leq n\|\tfrac{1}{2^m}*x\| = \tfrac{n}{2^m}\|x\|,$$
	by Corollary \ref{BMC:1} and the triangle inequality. Similarly, 
	\begin{multline*}
		\|x - m\cdot (\tfrac{1}{2^m}*x)\| = \|2^m \cdot (\tfrac{1}{2^m}* x) - m\cdot (\tfrac{1}{2^m}*x)\| = \|(2^m - n) \cdot (\tfrac{1}{2^m} * x)\| \leq \\ \leq (2^m - n)\|\tfrac{1}{2^m}*x\| = (1 - \tfrac{n}{2^m})\|x\|,
	\end{multline*}
	by Lemma \ref{BML:4}, Corollary \ref{BMC:1}, the triangle inequality and $\IZ$-module properties of $H_{x}.$ Assume that at least one of the above inequalities is strict. Then the triangle inequality ensures that
	$$\|x\| \leq \|m \cdot (\tfrac{1}{2^m}*x)\| + \|x - m\cdot (\tfrac{1}{2^m}*x)\| < \tfrac{n}{2^m}\|x\| + (1-\tfrac{n}{2^m})\|x\| = \|x\|,$$
	which is a contradiction. Hence, our assumption is wrong and both inequalities for element $n \cdot (\frac{1}{2^m}*x)$ hold with equality. Then the element $\frac{n}{2^m}*x$ coincides with $n \cdot (\frac{1}{2^m}*x)$ by the strict convexxity of the metric space $(G,d).$
\end{proof}

\begin{lemma}\label{BML:7}
	For all $x \in G$ and $t \in \IZ[\frac{1}{2}] \cap [0,1]$ we have $t \circ_x x = t * x.$
\end{lemma}
\begin{proof}
	Let $t = \frac{m}{2^n} \in \IZ[\frac{1}{2}]$, $m \leq 2^n.$ Lemma \ref{BML:4} ensures that $2^n \circ_x (\frac{1}{2^n} * x) = x$, which implies $\frac{1}{2^n} *x = \frac{1}{2^n} \circ_x x$. Then we obtain
	$$\tfrac{m}{2^n} \circ_x x = m \circ_x (\tfrac{1}{2^n} \circ_x x) = m \circ_x (\tfrac{1}{2^n}*x) = \tfrac{m}{2^n}*x,$$
	by the above equality, $\IZ[\frac{1}{2}]$-module properties of $\circ_x$ on $H_x$ and Lemma \ref{BML:6}.
\end{proof}

We extend both the algebraic multiplication and the metric multiplication on $G$ to $\boldsymbol{\cdot}: \mathbb{R} \times G \to G$, $\boldsymbol{\cdot}: (t,x) \mapsto tx$ by setting $tx = [t] \cdot x + \{t\} * x$ for all $t \in \mathbb{R}$ and $x \in G$. We will call this extension the {\em real multiplication.} Witness that the results of this section imply the following propoistion.

\begin{proposition}\label{BMP:3}
	For all $x \in G$ and $t \in \IZ[\frac{1}{2}]$ we have $tx = t \circ_x x.$
\end{proposition}
\begin{proof}
	Since $[t] \in \IZ$ and $\{t\} \in \IZ[\frac{1}{2}]\cap [0,1]$, we have
	$$tx = [t] \cdot x + \{t\} * x = [t] \circ_x x + \{t\} \circ_x x = ([t] + \{t\}) \circ_x x = t \circ_x x,$$
	by defenition of the real multiplication, Lemma \ref{BML:7} and $\IZ[\frac{1}{2}]$-module properties of $\circ_x$ on $H_x$.
\end{proof}

\section{The real multiplication is a topological group homomorphism}\label{s5}
In this section, we start by proving some important properties of the real multiplication, introduced in the following Lemmas.

\begin{lemma}\label{BML:8}
	For all $t, v \in \mathbb{R}$ and $x \in G$ we have $\|tx - vx\| \leq |t-v| \cdot \|x\|.$
\end{lemma}
\begin{proof}
	We lose no generality by assuming $t < v$. If $[t] = [v]$, we immediately obtain
	$$\|tx - vx\| = d([t]x + \{t\}*x, [v]x + \{v\}*x) = d(\{t\}*x, \{v\}*x) = |\{t\} - \{v\}| \cdot\|x\| = |t-v|\cdot\|x\|,$$
	by the definition of the algebraic multiplication, translation invariance of the metric and Lemma \ref{BML:1}. If $[t] < [v]$, we similarly obtain
	\begin{multline*}
		\|tx - vx\| \leq \|tx - [t+1]x\| + \|[t+1]x - [v]x\| + \|[v]x - vx\| = \|x - \{t\}*x\| + \|([t+1] - [v])x\| + \|\{v\}*x\| \\ \leq (1-\{t\})\|x\| + ([v] - [t+1])\|x\| + \{v\}\|x\| = (v-t)\|x\| = |t-v| \cdot\|x\|,
	\end{multline*}
	by the definition of the algebraic multiplication, triangle inequality, Lemma \ref{BML:1} and Corollary \ref{BMC:1}.
\end{proof}
Lemma \ref{BML:8} implies
\begin{lemma}
	For all $x \in G$ the function $\boldsymbol{\cdot}_x: \mathbb{R} \to G$, $\boldsymbol{\cdot}_x: t \mapsto tx$, is continuous.
\end{lemma}
\begin{lemma}\label{BML:9}
	For all $t,v \in \mathbb{R}$ and $x \in G$ we have $(t+v)x = tx + vx.$
\end{lemma}
\begin{proof}
	The continuity of the functions $\boldsymbol{\cdot}_x: \IR \to G$ and $+:G\times G\to G$ implies that the set $F=\{(t,v)\in\IR\times\IR:(t+v)x=tx+vx\}$ is closed in the real plane $\IR\times \IR$. Then by Proposition \ref{BMP:3} and $\IZ[\frac{1}{2}]$-module properties of $H_x$ the closed set $F$ contains the dense subset $\IZ[\frac12]\times\IZ[\frac12]$ of $\IR\times\IR$ and hence $F=\IR\times\IR$, witnessing that $(t+v)x=tx+vx$ for all $t,v\in\IR$. 
\end{proof}

\begin{definition}
	A map $\phi: G \to S$ between two topological groups $G$ and $S$ is called a {\em topological group homomorphism} if it is a continuous group homomorphism from $G$ to $S$. If in addition, $\phi^{-1}: S \to G$ is also a topological group homomorphism, then $\phi$ is called a {\em topological group isomorphism.}
\end{definition}

Lemma \ref{BML:8} and Lemma \ref{BML:9} imply
\begin{proposition}\label{BMP:4}
	For all $x \in G$ the map $\boldsymbol{\cdot}_x: \mathbb{R} \to G$, $\boldsymbol{\cdot}_x: t \mapsto  tx$ is a topological group homomorphism.
\end{proposition}
For all $x \in G$ by $\IR x$ we denote the image of $\IR$ under the map $\boldsymbol{\cdot}_x: \IR \to G$. 

\section{Every subgroup $\IR x$ is isomorphic to $\IR$.}\label{s6}
In this section, we shall show that the map $\boldsymbol{\cdot}_x: \IR \to G$ is actually a topological group isomorphism. To prove it we first show that $\IR x$ is not monothetic. This will be done by contradiction, using the following Lemmas.
\begin{definition}
	A topological group $G$ is called \emph{monothetic} if it contains a dense cyclic subgroup; that is, there exists $g \in G$ such that $\overline{ \langle g \rangle} = \overline{\{ng: n \in \IZ\}} = G.$ The set of all such $g\in G$ is denoted by $\operatorname{Gen}(G) = \{g \in G: \overline{\langle g \rangle} = G\}.$
\end{definition}
\begin{lemma}\label{BML:10}
	If $\IR x \subseteq G$ is monothetic, then there exists dense $A \subseteq \IR$ such that $Gen(\IR x) = \boldsymbol{\cdot}_x(A).$
\end{lemma}
\begin{proof}
	See \cite{BMM}.
\end{proof}
\begin{lemma}\label{BML:11}
	If $\IR x \subseteq G$ is monothetic, then there exists $tx \in \operatorname{Gen}(\IR x)$ such that $\|tx\| < |t|\cdot\|x\|$.
\end{lemma}
\begin{proof}
	 Assume to the contrary that for all $tx \in \operatorname{Gen}(\IR x)$ we have $\|tx\| \ge |t|\cdot\|x\|$. Then Lemma \ref{BML:8} ensures that $\|tx\| = |t|\cdot\|x\|$.  The continuity of the functions $\boldsymbol{\cdot}_x$ and $\|\cdot\|: G \to \IR$ implies that the set $F = \{t \in \mathbb{R}: \|tx\| = |t| \cdot\|x\|\}$ is closed in the real line $\IR$. Therefore, the closed set $F$ contains the set $\boldsymbol{\cdot}_x^{-1}(\operatorname{Gen}(\IR x))$, which is dense subset of $\IR$ by Lemma \ref{BML:10}. Hence, $F = \IR$, witnessing that $\|tx\| = |t| \cdot\|x\|$ for all $t \in \IR.$ But then for $tx \in \operatorname{Gen}(\IR x)$ the element $\frac{t}{2} x \not \in \overline{\langle tx \rangle}$, which is a contradiction. Hence, there exists $tx \in \operatorname{Gen}(\IR x)$ with $\|tx\| < |t|\cdot\|x\|$.
\end{proof}

\begin{lemma}\label{BML:12}
	If $\IR x \subseteq G$ is monothetic, then there exist $tx \in \operatorname{Gen}(\IR x)$ and distinct $a \in \IR x$, $b \in H_{tx}$ such that $2^na = tx = 2^nb$ for some $n \in \mathbb{N}.$
\end{lemma}
\begin{proof}
	Let $tx \in \operatorname{Gen}(\IR x)$ be the topological generator with $\|tx\| < |t| \cdot \|x\|$, which exists by Lemma \ref{BML:11}. Since there exists $n \in \mathbb{N}$ such that $2^{n-1} \leq t < 2^n$, we have $\frac{t}{2^n} \in [0,1] \subseteq \IR.$ Then for $a:=\frac{t}{2^n}x \in \IR x$ we have
	$$\|a\| = \|\tfrac{t}{2^n}x\| = \|\tfrac{t}{2^n}*x\| = \tfrac{t}{2^n}\|x\|,$$
	by definition of the real multiplication and Corollary \ref{BMC:1}. Since $\boldsymbol{\cdot}_x: \IR \to G$ is a group homomorphism, we have $2^na = 2^n\frac{t}{2^n} x = tx,$ by the definition of the algebraic multiplication. Now consider element $b := \frac{1}{2^n}(tx) \in H_{tx}$. Witness that
	$$\|b\| = \|\tfrac{1}{2^n}(tx)\| = \|\tfrac{1}{2^n}*(tx)\| = \tfrac{1}{2^n}\|tx\| < \tfrac{t}{2^n}\|x\|,$$
	by definition of the real multiplication, Corollary \ref{BMC:1} and the initial strict inequality on $\|tx\|.$	Lemma \ref{BML:4} ensures that $2^nb = 2^n(\frac{1}{2^n}(tx)) = tx.$ Since $\|b\| < \|a\|$, the elements $a \in \IR x$ and $b \in H_{tx}$ are distinct with $2^na = tx = 2^nb.$
\end{proof}

\begin{proposition}\label{BMP:5}
	For all $x \in G$ the subgroup $\IR x \subseteq G$ is not monothetic.
\end{proposition}
\begin{proof}
	Assume to the contrary that $\IR x$ is monothetic. Lemma \ref{BML:12} ensures that there exists $tx \in \operatorname{Gen}(\IR x)$ and distinct $a \in \IR x, b \in H_{tx}$ with $2^na=tx=2^nb$ for some $n \in \mathbb{N}$. Witness that $H_{tx}$ is a $\IZ[\frac{1}{2}]$-module with the algebraic multiplication by Proposition \ref{BMP:2} and hence $\langle tx \rangle \subseteq H_{tx}.$ This implies $\IR x \subseteq \overline{H_{tx}}$ by density of $\langle tx \rangle$ in $\IR x.$ It is well-known that the closure $\overline{H_{tx}}$ of the abelian group $H_{tx}$ is abelian.
	Since $a \in \IR x \subseteq \overline{H_{tx}}$ and $b \in H_{tx} \subseteq \overline{H_{tx}}$, the equality $2^na = 2^nb$ implies $2^n(a-b) = 0$ by commutativity of $+$ on $\overline{H_{tx}}.$ Since $G$ has no elements of order $2$, it has no elements of order $2^n$ by induction and hence the equality above implies $a = b$, which is a contradiction. 
\end{proof}

\begin{definition}
	A topological group $G$ is called \emph{a one-parameter group} if there exists a topological group homomorphism  $\phi: \IR \to G$ such that $G = \phi(\IR).$
\end{definition}
\begin{corollary}\label{BMC:2}
	For all $x \in G$ the metric subgroup $\IR x \subseteq G$ is isomorphic to $\IR$.
\end{corollary}
\begin{proof}
	It is known \cite{BMM} that the real line $\IR$ is the unique (up to topological group isomorphism) one-parameter metrizable topological group, which is not monothetic. $\IR x$ is a metrizable topological group with topology enduced by the translation-invariant metric. Moreover, $\IR x$ is one-parametric by Proposition  \ref{BMP:4} and not monothetic by Proposition \ref{BMP:5}. Hence, $\IR x$ must be isomorphic to $\IR$.
\end{proof}
\begin{corollary}\label{BMC:3}
	For all $x \in G$ the metric subgroup $\IR x$ is closed in $G$. 
\end{corollary}
\begin{proof}
	Corollary \ref{BMC:2} ensures that the metric subgroup $\IR x$ is locally compact. It is well-known (see e.g. \cite{GZ}) that locally compact subgroups of topological groups are closed.
\end{proof}
\begin{corollary}\label{BMC:4}
	If $K \subseteq G$ is a compact subgroup of $G$, then $K = \{0\}.$
\end{corollary}
\begin{proof}
	Suppose to the contrary that there exists $x \in K$ such that $x \not = 0.$ Since $ \langle x \rangle = \{nx : n \in \IZ\} = \IZ x := \boldsymbol{\cdot}_x(\IZ) \subseteq \IR x$ is an image of $\mathbb{Z}$ under the topological group isomorphism $\boldsymbol{\cdot}_x: \mathbb{R} \to \IR x$, it is a closed subgroup of $\IR x$.  Corollary \ref{BMC:3} ensures that $\IR x$ is closed in $G$, which implies that $\IZ x$ is closed in $G$. Since $\IZ x = \langle x \rangle \subseteq K$, the group $\IZ x$ is a closed subgroup of a compact group $K$ and hence $\IZ x$ is itself compact. This is a contradiction since $\IZ x$ is an image of non-compact $\IZ \subseteq \IR$ under the topological group isomorphism $\boldsymbol{\cdot}_x: \IR \to \IR x.$
\end{proof}

\section{Proof of Theorem \ref{BMT:1}}\label{s:proof-main}
To prove the main result, we plan to use the result \cite{IWS} of Iwasawa, which is the following theorem.
\begin{theorem}\label{IWS}
	A connected locally compact topological group contains a compact invariant neighborhood of the identity if and only if the group is compact-by-abelian, i.e., contains a compact normal subgroup whose quotient group is abelian.
\end{theorem}
Now we finally present the proof of Theorem \ref{BMT:1}. Let us recall that it states the following: {\em every locally compact strictly convex metric group $(G, +, 0, d)$ is abelian.}
\begin{proof}
	The strictly convex metric space $G$ is geodesic by \cite{BM1} and hence connected. Since $G$ is locally compact we can pick an open neighborhood $U$ of identity $0 \in G$ such that $\overline{U}$ is compact in $G$. Let $B \subseteq U$ be a closed ball with center in $0 \in G$. Since $B \subseteq \overline{U}$ is closed, $B$ is compact in $G$. Since the metric $d$ is  translation invariant, the ball $B$ is a compact invariant neighborhood of the identity.  Then $G$ contains a compact normal subgroup $K$ whose quotient group is abelian by Theorem \ref{IWS}. Corollary \ref{BMC:4} ensures that $K = \{0\}$. Hence $G/K = G$ is abelian.
\end{proof}
To conclude this section, we shall consider the following corollary concerning compactly finite-dimensional metric groups.
	\begin{definition}
	A metric group $G$ is called {\em compactly finite-dimensional} if 
	$$\operatorname{co-dim}(G) = \operatorname{sup}\{\operatorname{dim}(K): \text{$K$ is compact subspace of $G$}\}$$
	is finite.
\end{definition}
\begin{corollary}
	Every compactly finite-dimensional strictly convex metric group is abelian.
\end{corollary}
\begin{proof}
	It is known \cite{BZ} that every compactly finite-dimensional path-connected topological group is locally-compact and hence satisfies Theorem \ref{BMT:1}.
\end{proof}

\end{document}